\documentclass[12pt]{article}

\usepackage{amsfonts,amsmath,amssymb,latexsym,xcolor,mathrsfs,tikz,breqn,extarrows}
\usepackage{amsthm,authblk}
\usepackage{multirow}
\usepackage{mathtools}
\usepackage{calc}
\usepackage{pgfplots}
\pgfplotsset{compat=1.18}
\usetikzlibrary{patterns,arrows,decorations.pathreplacing}
\usepackage{diagbox}
\usepackage{makecell}
\usepackage{float}
\usepackage{enumitem}

\def\0{\emptyset}

\newtheorem{theorem}{Theorem}[section]

\newtheorem{lemma}[theorem]{Lemma}

\newtheorem{prop}[theorem]{Proposition}

\newcommand{\sat}{\mathrm{sat}}
\newcommand{\Sat}{\mathrm{Sat}}

\numberwithin{equation}{section}

\usepackage{algorithm}
\usepackage{algpseudocode}
\usepackage{graphicx}
\usepackage{pstricks,pgf,subcaption,amssymb}

\begin{document}
\title{\bf
A note on the saturation number for unions of three cliques
	 }
\author[1]{Hanlai Lin}
\author[2,3]{Zhen He}
\author[4]{Yiduo Xu\thanks{Corresponding author. E-mail:\texttt{xyd23@mails.tsinghua.edu.cn}}\;}

\affil[1]{\small Department of General Education, Wuxi University, Wuxi 214105, P. R. China}
\affil[2]{\small School of Mathematics and Statistics, Beijing Jiaotong University, Beijing 100044, P.R. China..}
\affil[3]{\small Beijing Key Laboratory of Biological Big Data and Topological Statistics, Beijing Jiaotong University, Beijing, 100044, P.R. China.}
\affil[4]{\small Department of Mathematical Sciences, Tsinghua University, Beijing 100084, P.R. China.}

\date{}

\maketitle\baselineskip 16.3pt

\begin{abstract}
A graph $G$ is $F$-saturated if $G$ contains no copy of $F$ but $G+e$ contains a copy of $F$ for every missing edge $e$ of $G$. 
The saturation number $\sat(n,F)$ is the minimum number of edges in an $n$-vertex $F$-saturated graph. 
Motivated by a problem posed by Faudree, Ferrara, Gould, and Jacobson concerning $K_p\cup K_q\cup K_{q+1}$, we determine the saturation number and the unique extremal graph for $K_p\cup K_q\cup K_r$ whenever $2\le p\le q<r<p+q$ and $n$ is sufficiently large. 

Together with the previously known results for $r\ge p+q$ and for $r=q$, this completes the determination of the saturation number and the extremal graphs for unions of three cliques, for all sufficiently large $n$.
\end {abstract}


\section{Introduction}

All graphs in this paper are finite, simple, and undirected.
For a graph $G$, let $V(G)$, $E(G)$, $|G|$, $e(G)$ denote its vertex set, edge set, order and size, respectively.
For $v\in V(G)$, write $N_G(v)$ and $N_G[v]$ for its open and closed neighbourhoods, and write $d_G(v)=|N_G(v)|$. 
The minimum degree of $G$ is $\delta(G)=\min\{d_G(u):u\in V(G)\}$. 
We omit the subscript when the underlying graph is clear.
For two graphs $G_1,G_2$, let $G_1 \cup G_2$ denote the vertex-disjoint union of $G_1$ and $G_2$. Their join $G_1 \vee G_2$ is obtained from $G_1 \cup G_2$ by adding every edge between $V(G_1)$ and $V(G_2)$.
We write $I_m$ for an independent set of order $m$.

Given graphs $G$ and $F$, we say that $G$ is \textit{$F$-free} if $G$ contains no copy of $F$.
The graph $G$ is \textit{$F$-saturated} if $G$ is $F$-free but $G+e$ contains a copy of $F$
for every $e \in E(\overline{G})$.
The \textit{saturation number} of $F$ is denoted by
$$\sat(n,F) =\min \{ e(G):G  \; \text{is}  \; F\text{-saturated and } |G|=n  \} \,.$$
We denote the family of extremal graphs by
$$\Sat(n,F)=\{G: |G|=n,\ e(G)=\sat(n,F),\ \text{and }G\text{ is }F\text{-saturated}\}.$$

For $t\ge 1$ and $2\le p_1\le \cdots\le p_t$, define
$$H_t(n;p_1,p_2,...,p_t)\cong K_{p_1-2}\vee(K_{p_2+1}\cup...\cup K_{p_t+1}\cup I_{n-t+3-\sum_{i=1}^{t}p_i}).$$
The saturation number was introduced by Erd\H{o}s, Hajnal and Moon~\cite{EHM}, who proved that $H_1(n;p)$ is the unique extremal graph for $K_p$.
Saturation for vertex-disjoint unions of cliques was developed by K\'aszonyi and Tuza~\cite{KT}. They proved $H_t(n;2,2,...,2)$ is the unique extremal graph for $tK_2$.
Faudree, Ferrara, Gould and Jacobson~\cite{FFGJ} subsequently determined $H_t(n; p,p,\ldots,p)\in Sat(n,tK_p)$ and $H_2(n;p,q)$ is the unique extremal graph for $K_p\cup K_q$ for sufficiently large $n$ and asked whether $H_t(n; p,p,\ldots,p)$ is unique for $tK_p$. 
Zhu, Hao and He~\cite{ZHH} answered this question and proved the uniqueness for the extremal graph for $tK_p$.

Chen and Yuan~\cite{CY} proved that $H_t(n; p,q,\ldots,q)\in Sat(n,K_p\cup(t-1)K_q)$ and $H_3(n;p,q,r)$ is the unique extremal graph for $K_p\cup K_q\cup K_r$ when $2\le p\le q\le r$, $r\ge p+q$ and $n$ is sufficiently large. In the case $q<r<p+q$, however, $H_3(n;p,q,r)$ contains $K_p\cup K_q\cup K_r$ and hence is not saturated. More generally,  Li, Hao, He and Zhu~\cite{LiHHZ} proved that $H_t(n;p_1,p_2,...,p_t)$ is $K_{p_1}\cup...\cup K_{p_t}$-saturated if and only if, for every $2\le i\le t-1$, either $p_{i+1}-p_i\ge p_1$ or $p_{i+1}=p_{i}$. They also determined $H_4(n;p_1,p_2,p_3,p_4)$ is extremal for $K_{p_1}\cup...\cup K_{p_4}$, under certain parameter restrictions

Faudree, Ferrara, Gould and Jacobson~\cite{FFGJ} proposed further investigation of the smallest case when $H_t(n;p_1,$ $p_2,...,p_t)$ is not $K_{p_1}\cup...\cup K_{p_t}$-saturated and they asked the problem for the saturation number of $K_p\cup K_{q}\cup K_{q+1}$. 
Our main theorem provides a strengthened version of the question raised in~\cite{FFGJ} by determining the exact saturation number and the unique extremal graph of $K_p\cup K_q\cup K_r$ when $2\le p\le q< r<p+q$ and $n$ is large enough. 
Together with the results in~\cite{CY} for $r\ge p+q$ and $r=q$, this completes the determination of the saturation number and the extremal graphs for unions of three cliques when $n$ is sufficiently large.

\begin{theorem}\label{thm:main}
Let \(2\le p\le q<r<p+q\) and \(n>3(p-2)+(q+r)(q+r+1)\).
Then
\[
\sat(n,K_p\cup K_q\cup K_r)
=(p-2)(n-p+2)+\binom{p-2}{2}+\binom{q+r+1}{2}.
\]
Moreover,
$\Sat(n,K_p\cup K_q\cup K_r) =\{H_2(n,p,q+r)\}.$
\end{theorem}

\section{Preliminaries}
Let $G_0=H_2(n,p,q+r)=K_{p-2}\vee\bigl(K_{q+r+1}\cup I_{n-p-q-r+1}\bigr).$
One checks that $G_0$ is $K_p\cup K_q\cup K_r$-saturated and
\[
e(G_0)
=(p-2)(n-p+2)+\binom{p-2}{2}+\binom{q+r+1}{2}.
\]
This provides the upper bound in Theorem~\ref{thm:main}.

We now state two lemmas that will be used in the proof of the lower bound.

\begin{lemma}\label{lem:min-degree}
Let $2\le p\le q<r$ and
\(
n>3(p-2)+(q+r)(q+r+1).
\)
If $G$ is $K_p\cup K_q\cup K_r$-saturated and
\begin{equation*}
e(G)\le
(p-2)(n-p+2)+\binom{p-2}{2}+\binom{q+r+1}{2}=e(G_0),
\end{equation*}
then $\delta(G)=p-2$.

Moreover, if $d(v)=p-2$ and $M=N(v)$, then $G[M]\cong K_{p-2}$ and the induced bipartite graph $G[M,V(G)\setminus M]$ is complete. In particular, $e(G[V(G)\setminus M]) \le \binom{q+r+1}{2}$.
\end{lemma}

\begin{proof}
It is clear that $\delta(G) \geq p-2$, otherwise $G$ is not $K_p\cup K_q\cup K_r$-saturated.
If $\delta(G) \geq p-1$, adapting the same argument of Claim 2.2 in \cite{FFGJ} we have
\[
2e(G)\ge (n-p+1)(p-1)+(n-p)(p-2)+2\binom{p-2}{2}+(p-1) > 2e(G_0),
\]
a contradiction.
Thus $\delta(G)=p-2$. Let $v \in V(G)$ has minimum degree and $M=N(v)$. For any $ u \not \in N[v]$, $G+uv$ creates a new copy of $K_p\cup K_q\cup K_r$ where $uv$ can only belongs to $K_p$. So $u$ is completely joind to $M$ and $G[M] \cong K_{p-2}$, which establishes the structural claim. The edge hypothesis forces
$e\bigl(G[V(G)\setminus M]\bigr)\le\binom{q+r+1}{2}$.
\end{proof}

\begin{lemma}\label{lem:separation}
Let $k \geq 0$ and $F_1,\ldots,F_m$ be finite sets.
Suppose that for every $i\in[m]$ there are $R_i\subseteq F_i$, $|R_i|=k$ and a red-blue colouring of $U\setminus R_i$ where $U=\bigcup_{i=1}^{m}F_i$ such that
\begin{enumerate}[label=\textup{(\alph*)}]
\item $F_i\setminus R_i$ contains both colours;
\item $F_j\setminus R_i$ is monochromatic for every $j\ne i$.
\end{enumerate}
Then $|U|\ge m+k+1.$
\end{lemma}

\begin{proof}
We use induction on $m$.
The hypothese associated with $F_i$ for $i\in[m]$ gives $|F_i|\ge k+2$.
For $m=1$, $|U|=|F_1|\ge k+2=m+k+1$.

Suppose $m\ge2$, and consider the colouring associated with $F_m$.
Let $X$ and $Y$ be its red and blue classes of $U\setminus R_m$.
Partition $[m-1]$ into $I_X=\{j<m:F_j\setminus R_m\subseteq X\}$ and $I_Y=\{j<m:F_j\setminus R_m\subseteq Y\}$.
Write $a=|I_X|$ and $b=|I_Y|$, so $a+b=m-1$.

If both classes are nonempty, each subfamily inherits the hypotheses of the lemma.
By induction, $\left|\bigcup\limits_{j\in I_X}F_j\right|\ge a+k+1$ and $\left|\bigcup\limits_{j\in I_Y}F_j\right|\ge b+k+1$.
The two unions intersect only inside $R_m$, so their intersection has size at most $k$.
Therefore $|U|\ge(a+k+1)+(b+k+1)-k=m+k+1.$
Otherwise we can assume $I_Y=\varnothing$, induction applied to $F_1,\ldots,F_{m-1}$ gives $\left|\bigcup_{j<m}F_j\right|\ge m+k.$
Choose a blue element $y\in F_m\setminus R_m$, then $y \not \in \bigcup_{j<m}F_j$ and hence $|U|\ge m+k+1$. This complete the proof.
\end{proof}

\section{The lower bound}

Let $G$ be an minimal $K_p\cup K_q\cup K_r$-saturated graph satisfying the hypotheses of Theorem~\ref{thm:main}.
Then Lemma~\ref{lem:min-degree} applies. Partition $V(G)$ into three sets $M, H, I$, where $M \cong K_{p-2}$ such that the induced bipartite graph $G[M,H \cup I]$ is complete and $I$ is the set of all isolated vertices in $G[H \cup I]$. Since $\delta(G) = p-2$, we have $I\ne\varnothing$ and we can fix $v \in I$. Furthermore, $e(H) \le\binom{q+r+1}{2}$ and $H$ is non-empty.

\begin{lemma}\label{lem:core}
The graph $H$ has the following properties.
\begin{enumerate}
\item $H$ contains no pairwise vertex-disjoint copies of $K_2,K_q,K_r$.
\item For every $z\in V(H)$, the graph $H-z$ contains vertex-disjoint copies of $K_q$ and $K_r$.
\end{enumerate}
\end{lemma}

\begin{proof}
If an edge $xy$, a $K_q$, and a $K_r$ were pairwise vertex-disjoint in $H$, then together with $M$ they would form a copy of $K_p\cup K_q\cup K_r$ in $G$, proving (i).

For (ii), since $G+vz$ creates a new copy of $K_p\cup K_q\cup K_r$, the edge $vz$ must belong to $K_p$ because $p\le q<r$. 
Therefore the remaining two components form disjoint $K_q,K_r$ in $H \cup I - \{v,z\}$, and all their vertices belong to $H$.
\end{proof}

\begin{prop}\label{prop:qplus}
If $H$ contains vertex-disjoint cliques
$A\cong K_{q+1}$ and $B\cong K_r$, then $H\cong K_{q+r+1}$.
\end{prop}

\begin{proof}
Let $X=V(H)\setminus(A\cup B).$ If $X \neq \varnothing$, then by Lemma~\ref{lem:core}(i), $E(X,A)=\varnothing$ and $X$ is independent.

We claim that every $x\in X$ satisfies $d_B(x)\le r-q-1$.
Otherwise choose $T\subseteq N_B(x)$ with $|T|=r-q$, and partition $M=M_1\cup M_2$ where $|M_1|=p-(r-q)-1$ and $|M_2|=r-q-1$.
Then $M_1\cup\{x\}\cup T\cong K_p$, $B\setminus T\cong K_q$, and $M_2\cup A\cong K_r$ are disjoint cliques in $G$, a contradiction.

Since $r<p+q$ and $p\le q$, $d_H(x)=d_B(x)\le r-q-1\le p-2<q-1.$
Suppose that $xb\in E(H)$ for some $b\in B$.
Lemma~\ref{lem:core}(ii), applied to $b$, gives disjoint cliques $Q\cong K_q$ and $R\cong K_r$ in $H-b$.
By $d_H(x)<q-1$, $x$ belongs to neither $Q$ nor $R$.
Thus the edge $xb$ is disjoint from $Q\cup R$, contrary to Lemma~\ref{lem:core}(i). Hence $E(X,B)=\varnothing$ and $X=\varnothing$.
Therefore $|V(H)|=q+r+1$. This means that $G$ is a subgraph of $G_0$. Since $G_0$ is $K_p\cup K_q\cup K_r$-saturated, we have $G=G_0$ and $H\cong K_{q+r+1}$.
\end{proof}

\begin{prop}\label{prop:rplus}
If $H$ contains vertex-disjoint cliques $A\cong K_q$ and $B\cong K_{r+1}$, then $H\cong K_{q+r+1}$.
\end{prop}

\begin{proof}
Let $X=V(H)\setminus(A\cup B).$ If $X \neq \varnothing$, then by Lemma~\ref{lem:core}(i), $E(X,B)=\varnothing$ and $X$ is independent.

Let $a\in A\cap N_H(x)$.
If $d_A(x)\le q-1$, then we have $d_{H-a}(x)=d_A(x)-1\le q-2$.
By Lemma~\ref{lem:core}(ii), $H-a$ contains disjoint $Q\cong K_q$ and $R\cong K_r$.
The degree bound shows that $x$ belongs to neither clique.
Hence the edge $xa$ is disjoint from $Q\cup R$, contrary to Lemma~\ref{lem:core}(i). Therefore $d_A(x)=q$. 

Partition $M=M_1\cup M_2$ where $|M_1|=p-(r-q)-1$ and $|M_2|=r-q-1$.
Then $M_2\cup\{x\}\cup A\cong K_r$ and $M_1\cup B\cong K_{p+q}$ are disjoint cliques in $G$, a contradiction.
Consequently $X=\varnothing$ and $|V(H)|=q+r+1$. The same argument as in the proof of Proposition~\ref{prop:qplus} shows that $H$ is complete.
\end{proof}

\begin{prop}\label{prop:nonexistence}
Let $q<r$, and let $H$ be a nonempty graph such that
\begin{enumerate}
\item $H$ contains no pairwise vertex-disjoint $K_2,K_q,K_r$;
\item $H$ contains no vertex-disjoint $K_{q+1},K_r$;
\item $H$ contains no vertex-disjoint $K_q,K_{r+1}$;
\item for every $z\in V(H)$, the graph $H-z$ contains vertex-disjoint $K_q,K_r$.
\end{enumerate}
Then no such graph $H$ exists.
\end{prop}

\begin{proof}
Write $h=|V(H)|$.
Property (iv) implies $h\ge q+r+1$ and gives a clique $B\cong K_r$ in $H$.
For each $b\in B$, define $F_b=\{x\in V(H)\setminus\{b\}: bx\notin E(H)\}$ and $U=\bigcup_{b\in B}F_b$.
Since $B$ is a clique, we have $U\subseteq V(H)\setminus B$ and $|U|\le h-r$.

Fix $b\in B$.
By (iv), $H-b$ contains disjoint cliques $P_b\cong K_q$ and $Q_b\cong K_r$.
Let $R_b=V(H)\setminus\bigl(P_b\cup Q_b\cup\{b\}\bigr)$, then $|R_b|=h-q-r-1=:k$.
By (i), $\{b\}\cup R_b$ is independent. Hence $R_b\subseteq F_b$.
Moreover, $F_b\cap P_b\ne\varnothing$, otherwise $P_b\cup\{b\}\cong K_{q+1}$ would be disjoint from $Q_b\cong K_r$, contrary to (ii).
Similarly, $F_b\cap Q_b\ne\varnothing$.

Colour $U\setminus R_b$ red on $P_b$ and blue on $Q_b$.
This is well-defined because $P_b\cup Q_b\cup\{b\}\cup R_b$ is a partition of $V(H)$ and $b\notin U$.
The argument above shows that $F_b\setminus R_b$ contains both colours.

Now let $c\in B\setminus\{b\}$, then $c\in P_b\cup Q_b$.
If $c\in P_b$, then $P_b \cup \{b\} \subseteq N_H[c]$.
Thus $F_c\setminus R_b\subseteq Q_b$.
It is therefore monochromatic blue.
If $c\in Q_b$, the symmetric argument gives
$F_c\setminus R_b\subseteq P_b$, which is monochromatic red.

The family $\{F_b:b\in B\}$ now satisfies Lemma~\ref{lem:separation} by taking $m=r$ and $k=h-q-r-1$. Consequently,
$|U|\ge r+(h-q-r-1)+1=h-q$.
Combining $|U|\ge h-q$ and $|U|\le h-r$ gives $h-q\le h-r$, contrary to $r>q$. This completes the proof.
\end{proof}
\vspace{0.5em}

\noindent \textbf{Proof of Theorem~\ref{thm:main}}.
The construction of $G_0$ gives the desired upper bound. For the lower bound, Lemma~\ref{lem:core} and Proposition~\ref{prop:nonexistence} show that either $H$ contains $K_{q+1}\cup K_r$ or $H$ contains $K_q\cup K_{r+1}$. By Proposition~\ref{prop:qplus} and Proposition~\ref{prop:rplus}, $H\cong K_{q+r+1}$ which forces $G=G_0$. This complete the proof.

\section*{Acknowledgement}

The authors used GPT-5.6 to find the proof of Lemma~\ref{lem:separation}, and for grammar and style checking. 
After using this tool, the authors reviewed and edited the content as needed and take full responsibility for the content of the manuscript.

The research of He was supported by Beijing Natural Science Foundation (No. QG26001),the National Natural Science Foundation of China (No.12401445) and the Talent Fund of Beijing Jiaotong University (No. 2024-003). The research of Xu was supported by China Scholarship Council (No. 202606210184). The research of Lin was supported by the Wuxi University Research Start-up Fund for High-level Talents (No. 550225090).

\end{document}